\documentclass[dvips,coll]{imsart2}

\RequirePackage[OT1]{fontenc}

\RequirePackage{amsthm,amsmath}

\startlocaldefs

\newcommand{\Z}        {{{\rm Z\!\! Z}}}
\newcommand{\E}        {{{\rm I\! E}}}
\def\1{\mbox{1\hspace{-.25em}I}}

\def\t{\theta}
\def\d{\delta}
\def\D{\Delta}
\def\Th{\Theta}
\def\e{\varepsilon}

\def\bR{\mathbb{R}}

\newtheorem{theorem}{Theorem}
\newtheorem{lemma}{Lemma}
\newtheorem{remark}{Remark}
\endlocaldefs
\usepackage{amssymb}
\usepackage{latexsym}
\usepackage{amsmath}
\usepackage{amsfonts}
\usepackage[pdftex]{graphicx}
\usepackage{amsfonts,amsmath,amstext,amscd,amsbsy}
\usepackage{amsfonts}
\usepackage{fancyhdr}
\usepackage{bbm}
\usepackage{picins}

\begin{document}

\begin{frontmatter}
\title{Improved Matrix Uncertainty Selector}
\runtitle{Matrix Uncertainty Selector}

\author{Mathieu Rosenbaum\thanksref{a}}
\and
\author{Alexandre B.~Tsybakov\thanksref{b,t2}}

\thankstext{t2}{Supported in part by ANR ``Parcimonie" and by
PASCAL-2 Network of Excellence.}

\address[a]{Universit\'e Pierre et Marie Curie, Paris-6, LPMA, case
courrier 188, 4 place Jussieu, 75252 Paris Cedex 05, France and
CREST}
\address[b]{CREST (ENSAE), 3, av.\ Pierre Larousse, 92240 Malakoff, France}


\runauthor{M.Rosenbaum and A.B.Tsybakov}

\contributor{???}{???}

%
%
%

\begin{abstract}
We consider the regression model with observation error in the
design:
\begin{eqnarray*}
y&=& X\theta^* + \xi,\\Z&=&X+\Xi.
\end{eqnarray*}
Here the random vector $y\in\bR^n$ and the random $n\times p$ matrix
$Z$ are observed, the $n\times p$ matrix $X$ is unknown, $\Xi$ is an
$n\times p$ random noise matrix, $\xi\in \bR^n$ is a random noise
vector, and $\theta^*$ is a vector of unknown parameters to be
estimated. We consider the setting where the dimension $p$ can be
much larger than the sample size $n$ and $\theta^*$ is sparse.
Because of the presence of the noise matrix $\Xi$, the commonly used
Lasso and Dantzig selector are unstable. An alternative procedure
called the Matrix Uncertainty (MU) selector has been proposed in
Rosenbaum and Tsybakov~(2010) in order to account for the noise. The
properties of the MU selector have been studied in Rosenbaum and
Tsybakov~(2010) for sparse $\theta^*$ under the assumption that the
noise matrix $\Xi$ is deterministic and its values are small. In
this paper, we propose a modification of the MU selector when $\Xi$
is a random matrix with zero-mean entries having the variances that
can be estimated. This is, for example, the case in the model where
the entries of $X$ are missing at random. We show both theoretically
and numerically that, under these conditions, the new estimator
called the Compensated MU selector achieves better accuracy of
estimation than the original MU selector.
\end{abstract}



\end{frontmatter}

\section{Introduction}\label{intro}

We consider the model
\begin{eqnarray}\label{0}
y&=& X\theta^* + \xi,\\Z&=&X+\Xi,\label{00}
\end{eqnarray}
where the random vector $y\in\bR^n$ and the random $n\times p$
matrix $Z$ are observed, the $n\times p$ matrix $X$ is unknown,
$\Xi$ is an $n\times p$ random noise matrix, $\xi\in \bR^n$ is a
random noise vector, $\theta^*=(\theta^*_1,\dots,\theta^*_p)\in
\Theta$ is a vector of unknown parameters to be estimated, and
$\Theta$ is a given subset of $\bR^p$.  We consider the problem of
estimating an $s$-sparse vector $\theta^*$ (i.e., a vector
$\theta^*$ having only $s$ non zero components), with $p$ possibly
much larger than $n$. If the matrix $X$ in \eqref{0}--\eqref{00} is
observed without error ($\Xi=0$), this problem has been recently
studied in numerous papers. The proposed estimators mainly rely on
$\ell_1$ minimization techniques. In particular, this is the case
for the widely used Lasso and Dantzig selector, see among others
Cand\`es and Tao~(2007), Bunea et al.~(2007a,b), Bickel et
al.~(2009), Koltchinskii~(2009), the book by B\"uhlmann and van de
Geer~(2011), the lecture notes by Koltchinskii~(2011), Belloni and
Chernozhukov~(2011) and the references cited therein.


However, it is shown in Rosenbaum and Tsybakov~(2010) that dealing
with a noisy observation of the regression matrix $X$ has severe
consequences. In particular, the Lasso and Dantzig selector become
very unstable in this context. An alternative procedure, called the
matrix uncertainty selector (MU selector for short) is proposed in
Rosenbaum and Tsybakov~(2010) in order to account for the presence
of noise $\Xi$. The MU selector $\hat \theta^{MU}$ is defined as a
solution of the minimization problem
\begin{equation}\label{mu_sel}
\min\{ |\t|_1: \,\, \t\in\Th, \,
\Big|\frac1{n}Z^T(y-Z\t)\Big|_\infty\le \mu |\t|_1 +\tau\},
\end{equation}
where $|\cdot|_{p}$ denotes the $\ell_p$-norm, $1\le p\le\infty$,
$\Th$ is a given subset of $\mathbb{R}^p$ characterizing the prior
knowledge about $\theta^*$, and the constants $\mu$ and $\tau$
depend on the level of the noises $\Xi$ and $\xi$ respectively. If
the noise terms $\xi$ and $\Xi$ are deterministic, it is suggested
in Rosenbaum and Tsybakov~(2010) to choose $\tau$ such that
$$
\Big|\frac1{n}Z^T\xi\Big|_\infty\le \tau,
$$
and to take $\mu=\delta(1+\delta)$ with $\d$ such that
$$
|\Xi|_\infty\le\d,
$$
where, for a matrix $A$, we denote by $|A|_\infty$ its componentwise
$\ell_\infty$-norm.

In this paper, we propose a modification of the MU selector for the
model where $\Xi$ is a random matrix with independent and zero mean
entries $\Xi_{ij}$ such that the sums of expectations
$$
\sigma_j^2\triangleq\frac{1}{n}\sum_{i=1}^n \E(\Xi_{ij}^2), \quad
1\leq j\leq p,
$$
are finite and admit data-driven estimators. Our main example where
such estimators exist is the model with data missing at random (see
below).  The idea underlying the new estimator is the following. In
the ideal setting where there is no noise $\Xi$, the estimation
strategy for $\theta^*$ is based on the matrix $X$. When there is
noise this is impossible since $X$ is not observed and so we have no
other choice than using $Z$ instead of $X$. However, it is not hard
to see that under the above assumptions on $\Xi$, the matrix
$Z^TZ/n$ appearing in \eqref{mu_sel} contains a bias induced by the
diagonal entries of the matrix $\Xi^T\Xi/n$ whose expectations
$\sigma_j^2$ do not vanish. If $\sigma_j^2$ can be estimated from
the data, it is natural to make a bias correction. This leads to a
new estimator $\hat \theta$ defined as a solution of the
minimization problem
\begin{equation}\label{cmu}
\min\{|\t|_1: \,\, \t\in\Th, \,
\Big|\frac{1}{n}Z^T(y-Z\t)+\widehat{D}\theta\Big|_\infty\le\mu|\t|_1
+\tau\},
\end{equation}
where $\widehat{D}$ is the diagonal matrix with entries
$\hat{\sigma}_j^2$, which are estimators of $\sigma_j^2$, and
$\mu\geq 0$ and $\tau\geq 0$ are constants that will be specified
later. This estimator $\hat \theta$ will be called the Compensated
MU selector. In this paper, we show both theoretically and
numerically that the estimator $\hat \theta$ achieves better
performance than the original MU selector~$\hat \theta^{MU}$. In
particular, under natural conditions given below, the bounds on the
error of the Compensated MU selector decrease as $O(n^{-1/2})$ up to
logarithmic factors as $n\to\infty$, whereas for the original MU
selector $\hat \theta^{MU}$ the corresponding bounds do not decrease
with $n$ and can be only small if the noise $\Xi$ is small.

\begin{remark}\label{rem:1}
The problem (\ref{cmu}) is equivalent to
\begin{equation}\label{i4}
\min_{(\theta, u)\in W(\mu,\tau)} |\theta|_1,
\end{equation}
where
\begin{equation}\label{i5}
W(\mu,\tau)=\Big\{(\theta,u) \in \Theta \times{\bf R}^p:
\left|\frac1{n}Z^T(y-Z\theta)+ \hat D \theta + u\right|_\infty \le
\tau, \  |u|_\infty \le \mu |\theta|_1\Big\},
\end{equation}
with the same $\mu$ and $\tau$ as in (\ref{cmu}) (see the proof in
Section \ref{sec:preuve}). This simplifies in some cases the
computation of the solution.
\end{remark}

An important example where the values $\sigma_j^2$ can be estimated
is given by the model with missing data. Assume that the elements
$X_{ij}$ of the matrix $X$ are unobservable, and we can only observe
\begin{equation}\label{ex1}
\tilde Z_{ij} = X_{ij}\eta_{ij},\quad i=1,\dots,n,\ j=1,\dots,p,
\end{equation}
where for each fixed $j=1,\dots,p$, the factors $\eta_{ij},
i=1,\dots,n,$ are i.i.d. Bernoulli random variables taking value 1
with probability $1-\pi_j$ and 0 with probability $\pi_j$,
$0<\pi_j<1$. The data $X_{ij}$ is missing if $\eta_{ij}=0$, which
happens with probability $\pi_j$. We can rewrite (\ref{ex1}) in the
form
\begin{equation}\label{ex2}
Z_{ij} = X_{ij}+ \Xi_{ij},
\end{equation}
where $Z_{ij}={\tilde Z}_{ij}/(1-\pi_j)$,
$\Xi_{ij}=X_{ij}(\eta_{ij}-(1-\pi_j))/(1-\pi_j)$. Thus, we can
reduce the model with missing data (\ref{ex1}) to the form
(\ref{00}) with a matrix $\Xi$ whose elements $\Xi_{ij}$ have zero
mean and variance $X_{ij}^2\pi_j/(1-\pi_j)$. So,
\begin{equation}\label{ex3}
\sigma_j^2 = \frac1{n}\sum_{i=1}^n X_{ij}^2 \,
\frac{\pi_j}{1-\pi_j}.
\end{equation}
In Section~\ref{sec:stoch} below, we show that when the $\pi_j$ are known, the $\sigma_j^2$ admit
good data-driven estimators $\hat\sigma_j^2$. If the $\pi_j$ are
unknown, they can be readily estimated by the empirical frequencies
of 0 that we further denote by $\hat \pi_j$. Then the $Z_{ij}={\tilde
Z}_{ij}/(1-\pi_j)$ appearing in (\ref{ex2}) are not available and
should be replaced by $Z_{ij}={\tilde Z}_{ij}/(1-\hat\pi_j)$. This
slightly changes the model and implies a minor modification of the
estimator (cf. Section~\ref{sec:stoch}).²²

\section{Definitions and notation}\label{sec:def}

Consider the following random matrices
$$M^{(1)}=\frac{1}{n}X^T\Xi,~M^{(2)}=\frac{1}{n}X^T\xi,~M^{(3)}=\frac{1}{n}\Xi^T\xi,$$
$$M^{(4)}=\frac{1}{n}(\Xi^T\Xi-\text{Diag}\{\Xi^T\Xi\}),
~M^{(5)}=\frac{1}{n}\text{Diag}\{\Xi^T\Xi\}-D,
$$
where $D$ is the diagonal matrix with diagonal elements
$\sigma_j^2$, $j=1,\dots,p$, and for a square matrix $A$, we denote
by $\text{Diag}\{A\}$ the matrix with the same dimensions as $A$,
the same diagonal elements as $A$ and all off-diagonal elements
equal to zero.

Under conditions that will be specified below, the entries of the
matrices $M^{(k)}$ are small with probability close to 1. Bounds on
the $\ell_\infty$-norms of the matrices $M^{(k)}$ characterize the
stochastic error of the estimation. The accuracy of the estimators
is determined by these bounds and by the properties of the Gram
matrix
$$
\Psi \triangleq \frac1{n}X^TX.
$$
For a vector $\t$, we denote by $\t_J$ the vector in $\bR^p$ that
has the same coordinates as $\t$ on the set of indices $J\subset
\{1,\hdots,p\}$ and zero coordinates on its complement~$J^c$. We
denote by $|J|$ the cardinality of $J$.

To state our results in a general form, we follow Gautier and
Tsybakov~(2011) and introduce the sensitivity characteristics
related to the action of the matrix $\Psi$ on the cone
$$
C_{J}\triangleq\left\{\Delta\in\bR^p:\ |\Delta_{J^c}|_1\le
|\Delta_{J}|_1 \right\},
$$
where $J$ is a subset of $\{1,\hdots,p\}$. For $q\in[1,\infty]$ and
an integer $s\in[1,p]$, we define the {\em $\ell_q$ sensitivity} as
follows:
$$
\kappa_{q}(s)\triangleq\min_{J: \ |J|\le s} \left(\min_{\Delta\in
C_J:\ |\Delta|_q=1} \left|\Psi\Delta \right|_{\infty}\right).
$$
We will also consider the {\em coordinate-wise sensitivities}
\begin{equation*}
\kappa_{k}^*(s)\triangleq \min_{J: \ |J|\le s} \left(\min_{\Delta\in
C_J: \ \Delta_k= 1} \left|\Psi\Delta \right|_{\infty}\right),
\end{equation*}
where $\Delta_k$ is the $k$th coordinate of $\Delta$, $k=1,\dots,p$.
To get meaningful bounds for various types of estimation errors, we
will need the positivity of $\kappa_{q}(s)$ or $\kappa_{k}^*(s)$. As
shown in Gautier and Tsybakov~(2011), this requirement is weaker
than the usual assumptions related to the structure of the Gram
matrix $\Psi$, such as the Restricted Eigenvalue assumption and the
Coherence assumption. For completeness, we recall these two
assumptions.

\smallskip

{\bf Assumption RE($s$).} {\it Let $1\le s \le p$. There exists a
constant $\kappa_{\rm RE}(s)>0$ such that
$$
\min_{\D\in C_J\setminus \{0\} }\frac{|\D^T\Psi \D|}{|\D_J|_2^2} \ge
\kappa_{\rm RE}(s)
$$
for all subsets $J$ of $\{1,\dots,p\}$ of cardinality $|J|\le s$.}

\smallskip

{\bf Assumption C.} {\it All the diagonal elements of $\Psi$ are
equal to 1 and all its off-diagonal elements of $\Psi_{ij}$ satisfy
the coherence condition:
$\max_{i\neq j} |\Psi_{ij}|\le \rho$ for some $\rho<1$.}\\

Note that Assumption C with $\rho<(3s)^{-1}$ implies Assumption
RE($s$) with $\kappa_{\rm RE}(s)=\sqrt{1-3\rho s}$, see Bickel et al.~(2009) 
or Lemma 2 in Lounici~(2008). From Proposition~4.2 of Gautier and
Tsybakov~(2011) we get that, under Assumption C with $\rho<(2
s)^{-1}$,
\begin{equation}\label{k1}
\kappa_\infty(s)\ge 1-2\rho s,
\end{equation}
which yields the control of the sensitivities $\kappa_q(s)$ for all
$1\le q\le\infty$ since
\begin{equation}\label{k2}
\kappa_q(s) \ge (2s)^{-1/q}\kappa_\infty(s), \quad \forall \ 1\le
q\le\infty,
\end{equation}
by Proposition~4.1 of Gautier and Tsybakov~(2011). Furthermore,
Proposition~9.2 of Gautier and Tsybakov~(2011) implies that, under
Assumption RE($s$),
\begin{equation}\label{k3}
\kappa_{1}(s) \ge (4s)^{-1}\kappa_{\rm RE}(s),
\end{equation}
 and by
Proposition~9.3 of that paper, under Assumption RE($2s$) for any
$s\le p/2$ and any $1< q\le 2$, we have
\begin{equation}\label{k4}
\kappa_{q}(s) \ge C(q) s^{-1/q}\kappa_{\rm RE}(2s),
\end{equation}
where $C(q)=2^{-1/q-1/2} \big(1 +
\left(q-1\right)^{-1/q}\big)^{-1}$.

\section{Main results}\label{sec:main}

In this section, we give bounds on the estimation and prediction
errors of the Compensated MU selector. For $\varepsilon\ge 0$, we
consider the thresholds $b(\varepsilon)\ge 0$ and
$\delta_i(\varepsilon)\ge 0$, $i=1,\dots, 5$, such that
\begin{equation}\label{pr1}
\mathbb{P}\big(\max_{j=1,\dots,p}|\hat{\sigma}_j^2-\sigma_j^2|\geq
b(\e)\big)\leq \varepsilon,
\end{equation}
and
\begin{equation}\label{pr2}
\mathbb{P}(|M^{(i)}|_{\infty}\geq \delta_i(\varepsilon))\leq
\varepsilon,\quad i=1,\ldots,5.
\end{equation}
Define
$$
\mu(\varepsilon)=\delta_1(\varepsilon)+\delta_4(\varepsilon)
+\delta_5(\varepsilon)+b(\varepsilon),\quad
\tau(\e)=\delta_2(\varepsilon)+ \delta_3(\varepsilon),
$$
and
$\mathcal{A}(\varepsilon)=\mathcal{A}(\mu(\varepsilon),\tau(\e)),$
where
\begin{equation}\label{premu_sel2}
\mathcal{A}(\mu,\tau)\triangleq\Big\{\theta\in\Theta: \,
\Big|\frac{1}{n}Z^T(y-Z\theta)+\widehat D\theta\Big|_{\infty}\leq
\mu|\theta|_1+\tau\Big\},\quad \forall \ \mu,\tau\ge0,
\end{equation} and $\Theta$ is a given subset of $\mathbb{R}^p$. For
$\varepsilon\ge 0$, the Compensated MU selector is defined as a
solution of the minimization problem
\begin{equation}\label{mu_sel2}
\min\{ |\t|_1: \,\, \t\in\mathcal{A}(\varepsilon)\},
\end{equation} We have the
following result.

\begin{theorem}\label{t1}
Assume that model (\ref{0})--(\ref{00}) is valid with an $s$-sparse
vector of parameters $\t^*\in \Th$, where $\Th$ is a given subset of
$\mathbb{R}^p$. For $\varepsilon\ge 0$, set
$$
\nu(\varepsilon)=2\big(\mu(\varepsilon)
+\delta_1(\varepsilon)\big)|\theta^*|_1+2\tau(\e).
$$
Then, with probability at least $1-6\varepsilon$, the set
$\mathcal{A}(\varepsilon)$ is not empty and for any solution
$\hat\t$ of \eqref{mu_sel2} we have
\begin{eqnarray}\label{t1:1}
&&|\hat\t-\t^*|_q\le \frac{\nu(\varepsilon)}{\kappa_q(s)}\,, \quad
\forall \ 1\le q\le \infty,
\\
&&|\hat\t_k-\t^*_k| \le
\frac{\nu(\varepsilon)}{\kappa_k^*(s)}\,,\quad \forall \ 1\le k \le
p, \label{t1:2}
\\
&&\frac{1}{n}|X(\hat\t-\t^*)|_2^2\le\min\Big\{
\frac{\nu^2(\varepsilon)}{\kappa_1(s)},\,
2\nu(\varepsilon)|\theta^*|_1 \Big\}\,.\label{t1:3}
\end{eqnarray}
\end{theorem}

The proof of this theorem is given in Section \ref{sec:preuve}.

\smallskip

Note that (\ref{t1:3}) contains a bound on the prediction error
under no assumption on $X$:
$$
\frac{1}{n}|X(\hat\t-\t^*)|_2^2\le 2\nu(\varepsilon)|\theta^*|_1 \,.
$$
The other bounds in Theorem \ref{t1} depend on the sensitivities.
Using (\ref{k1}) -- (\ref{k4}) we obtain the following corollary of
Theorem \ref{t1}.

\begin{theorem}\label{t2}
Let the assumptions of Theorem~\ref{t1} be satisfied. Then, with
probability at least $1-6\varepsilon$, for any solution $\hat\t$ of
\eqref{mu_sel2} we have the following inequalities. \vspace{1mm}

(i) Under Assumption RE($s$):
\begin{eqnarray}\label{t2:1}
|\hat\t-\t^*|_1&\le& \frac{4\nu(\varepsilon) s}{\kappa_{\rm
RE}(s)}\,,
\\
\frac{1}{n}|X(\hat\t-\t^*)|_2^2&\le& \frac{4\nu^2(\varepsilon)
s}{\kappa_{\rm RE}(s)}\,.\label{t2:2}
\end{eqnarray}
(ii) Under Assumption RE($2s$), $s\le p/2$:
\begin{eqnarray}\label{t2:3}
|\hat\t-\t^*|_q&\le& \frac{4\nu(\varepsilon)s^{1/q}}{\kappa_{\rm
RE}(2s)}\,, \quad \forall \,1< q\le 2.
\end{eqnarray}
(iii) Under Assumption C with $\rho<\frac1{2s}$:
\begin{eqnarray}\label{t2:4} |\hat\t-\t^*|_q&<&
\frac{(2s)^{1/q}\nu(\varepsilon)}{1-2\rho s} \,, \quad \forall \
1\le q\le \infty,
\end{eqnarray}
where we set $1/\infty =0$.
\end{theorem}


\smallskip

If the components of $\xi$ and $\Xi$ are subgaussian, the values
$\delta_i(\varepsilon)$ are of order $O(n^{-1/2})$ up to logarithmic
factors, and the value $b(\varepsilon)$ is of the same order in the
model with missing data (see Section \ref{sec:stoch}). Then, the
bounds for the Compensated MU selector in Theorem~\ref{t2} are
decreasing with rate $n^{-1/2}$ as $n\to \infty$. This is an
advantage of the Compensated MU selector as compared to the original
MU selector $\hat \theta^{MU}$, for which the corresponding bounds
do not decrease with $n$ and can be small only if the noise $\Xi$ is
small (cf. Rosenbaum and Tsybakov~(2010)).

If the matrix $X$ is observed without error ($\Xi=0$), then
$\mu(\varepsilon)=0$, $\delta_i(\varepsilon)=0, \ i\ne 2$, and the
Compensated MU selector coincides with the Dantzig selector. In this
particular case, the results (ii) and (iii) of Theorem~\ref{t2}
improve, in terms of the constants or the range of validity, upon
the corresponding bounds in Bickel et al.~(2009) and Lounici~(2008).


\section{Control of the stochastic error terms}\label{sec:stoch}

Theorems \ref{t1} and \ref{t2} are stated with general thresholds
$\delta_i(\varepsilon)$ and $b(\varepsilon)$, and can be used both
for random or deterministic noises $\xi, \Xi$ (in the latter case,
$\varepsilon=0$) and random or deterministic $X$. In this section,
considering $\varepsilon>0$ we first derive the values
$\delta_i(\varepsilon)$ for random $\xi$ and $\Xi$ with subgaussian
entries, and then we specify $b(\varepsilon)$ and the matrix
$\widehat D$ for the model with missing data. Note that, for random
$\xi$ and  $\Xi$, the values $\delta_i(\varepsilon)$ and
$b(\varepsilon)$ characterize the stochastic error of the estimator.

\subsection{Thresholds $\delta_i(\varepsilon)$ under subgaussian noise}

Recall that a zero-mean random variable $W$ is said to be
$\gamma$-subgaussian ($\gamma>0$) if, for all $t\in\mathbb{R}$,
\begin{equation}\label{subgauss}
\E[\text{exp}(tW)]\leq
\text{exp}(\gamma^2t^2/2).
\end{equation}
In particular, if $W$ is a zero-mean gaussian or bounded random
variable, it is subgaussian. A zero-mean random variable $W$ will be
called $(\gamma, t_0)$-subexponential if there exist $\gamma>0$ and
$t_0>0$ such that
\begin{equation}\label{subexp}
\E[\text{exp}(tW)]\leq \text{exp}(\gamma^2t^2/2), \quad \forall \
|t|\le t_0.
\end{equation}
Let the noise terms $\xi$ and $\Xi$ satisfy the following
assumption.

\smallskip


{\bf Assumption N.} {\it Let $\gamma_{\Xi}>0$, $\gamma_{\xi}>0$. The
entries $\Xi_{ij}$, $i=1,\dots,n, \ j=1,\dots,p,$ of the matrix $\Xi$
are zero-mean $\gamma_{\Xi}$-subgaussian random variables, the $n$
rows of $\Xi$ are independent, and $\ \E(\Xi_{ij}\Xi_{ik})=0$ for
$j\ne k$, $i=1,\dots,n$. The components $\xi_i$ of the vector $\xi$ are
independent zero-mean $\gamma_{\xi}$-subgaussian random variables
satisfying $\ \E(\Xi_{ij}\xi_i)=0$, $i=1,\dots,n,\ j=1,\dots,p$.}

\smallskip

Assumption N implies that the random variables $\Xi_{ij}\xi_{i}$,
$\Xi_{ij}\Xi_{ik}$ are subexponential. Indeed, if two random
variables $\zeta$ and $\eta$ are subgaussian, then for some $c>0$ we
have $\E \exp(c\zeta\eta)<\infty$, which implies that (\ref{subexp})
holds for $W=\zeta\eta$ with some $\gamma,t_0$ whenever
$\E(\zeta\eta)=0$, cf., e.g., Petrov~(1995), page 56.

Next, $\zeta_j\triangleq(1/n)\sum_{i=1}^n\Xi_{ij}^2-\sigma_j^2$ is a
zero-mean subexponential random variable with variance $O(1/n)$. It
is easy to check that (\ref{subexp}) holds for $W=\zeta_j$ with
$\gamma= O(1/\sqrt{n})$ and $t_0=O(n)$.

To simplify the notation, we will use a rougher evaluation valid
under Assumption N, namely that all $\Xi_{ij}\xi_{i}$,
$\Xi_{ij}\Xi_{ik}$ are $(\gamma_0, t_0)$-subexponential with the
same $\gamma_0>0$ and $t_0>0$, and all $\zeta_j$ are
$(\gamma_0/\sqrt{n}, t_0 n)$-subexponential. Here the constants
$\gamma_0$ and $t_0$ depend only on $\gamma_{\Xi}$ and
$\gamma_{\xi}$. For $0<\varepsilon<1$ and an integer $N$, set
$$
\bar \delta (\varepsilon,N) = \max\left(\gamma_0
\sqrt{\frac{2\log(N/\varepsilon)}{n}},\
\frac{2\log(N/\varepsilon)}{t_0n}\right)\,.
$$
\begin{lemma}\label{lem:subexp} Let Assumption N be satisfied, and
let $X$ be a deterministic matrix  with $\max_{1\le j\le
p}\frac1{n}\sum_{i=1}^n X_{ij}^2 \triangleq m_2$. Then for any
$0<\varepsilon<1$ the bound (\ref{pr2}) holds with
\begin{eqnarray}\label{eq:lem:subexp}
&&\delta_1(\varepsilon)=\gamma_{\Xi}
\sqrt{\frac{2m_2\log(2p^2/\varepsilon)}{n}}, \quad
\delta_2(\varepsilon)=\gamma_{\xi}
\sqrt{\frac{2m_2\log(2p/\varepsilon)}{n}},\\
&& \delta_3(\varepsilon)=\delta_5(\varepsilon)=\bar \delta
(\varepsilon,2p), \quad \delta_4(\varepsilon)=\bar \delta
(\varepsilon,p(p-1)).
\end{eqnarray}
\end{lemma}
\begin{proof} Use the union  bound and the facts that
$\mathbb{P}(W > \delta)\le \exp(-\delta^2/(2\gamma^2))$ for a
$\gamma$-subgaussian $W$, and $\mathbb{P}( \frac1{n}\sum_{i=1}^n
W_i> \delta)\le \max\big(\exp(-n\delta^2/(2\gamma^2)), \exp (-\delta
t_0n/2)\big)$ for a sum of independent $(\gamma,t_0)$-subexponential
$W_i$.
\end{proof}

\subsection{Data-driven $\widehat D$ and $b(\varepsilon)$ for the model with missing data}

Consider now the model with missing data (\ref{ex1}) and assume that
$X$ is non-random. Then we have ${\tilde Z}_{ij}^2 =
X_{ij}^2\eta_{ij}$, which implies:
$$
\E[{\tilde Z}_{ij}^2] = X_{ij}^2(1-\pi_j)\,, \quad j=1,\dots,p.
$$
Hence, ${\tilde Z}_{ij}^2\pi_j/(1-\pi_j)^2$ is an unbiased estimator
of $X_{ij}^2 \pi_j/(1-\pi_j)$. Then $\sigma^2_j$ defined in
(\ref{ex3}) is naturally estimated by
\begin{equation}\label{hsigm}
\hat \sigma^2_j = \frac1{n}\sum_{i=1}^n {\tilde Z}_{ij}^2\frac{
\pi_j}{(1-\pi_j)^2},
\end{equation}
 The matrix $\widehat D$ is then defined as a
diagonal matrix with diagonal entries $\hat \sigma^2_j$. It is not
hard to prove that $\hat \sigma^2_j$ approximates $\sigma^2_j$ in
probability with rate $O(n^{-1/2})$ up to a logarithmic factor.
For example, let the probability that the data is missing be the
same for all $j$: $\pi_1= \cdots = \pi_p\triangleq \pi_*$. Then
\begin{eqnarray*}
&&\mathbb{P}(|\hat{\sigma}_j^2-\sigma_j^2|\geq b) =
\mathbb{P}\left(\left|\frac1{n}\sum_{i=1}^n\Big( {\tilde
Z}_{ij}^2\frac{\pi_*}{(1-\pi_*)^2}-X_{ij}^2
\frac{\pi_*}{(1-\pi_*)}\Big)\right|\geq b\right) \\&& =
\mathbb{P}\left(\left|\frac1{n}\sum_{i=1}^n
Z_{ij}^2-\frac{X_{ij}^2}{(1-\pi_*)}\right|\geq
\frac{b}{\pi_*}\right) \le 2\exp\left(-
\frac{2nb^2(1-\pi_*)^4}{\pi_*^2m_4}\right),
\end{eqnarray*}
where we have used the fact that $0\le Z_{ij}^2\le
X_{ij}^2(1-\pi_*)^{-2}$, Hoeffding's inequality and the notation
$m_4 \triangleq \max_{1\le j\le p}\frac1{n}\sum_{i=1}^n X_{ij}^4 $.
This proves (\ref{pr1}) with
$$
b(\e) =
\frac{\pi_*}{(1-\pi_*)^{2}}\sqrt{\frac{m_4\log(2p/\e)}{2n}}\,.
$$
If $\pi_*$ is unknown, we replace it by the estimator $\hat\pi=
\frac1{np}\sum_{i,j} \mathrm{1}_{\{\tilde Z_{ij}=0\}}$, where $\mathrm{1}_{\{\cdot\}}$ denotes
the indicator function. Another difference is that $Z_{ij}={\tilde
Z}_{ij}/(1-\pi_j)$ appearing in (\ref{ex2}) are not available when
$\pi_j$'s are unknown. Therefore, we slightly modify the estimator
using ${\tilde Z}_{ij}$ instead of $Z_{ij}$; we define $\hat\theta$
as a solution of $\min\{ |\t|_1: \,\, \t\in\mathcal{\tilde
A}(\varepsilon)\}$ with
\begin{equation}\label{cmu_miss_data}
\mathcal{\tilde A}(\varepsilon)=\Big\{\theta\in\Theta: \,
\Big|\frac{1}{n}{\tilde Z}^T(y(1-\hat\pi) -{\tilde
Z}\theta)+\widehat D\theta\Big|_{\infty}\leq
\tilde\mu(\varepsilon)|\theta|_1+\tilde\tau(\e)\Big\},
\end{equation}
where $\tilde\mu(\varepsilon)$ and $\tilde\tau(\e)$ are suitably
chosen constants, $\tilde Z$ is the $n\times p$ matrix with entries
$\tilde Z_{ij}$, and $\widehat D$ is a diagonal matrix with entries
$\hat \sigma^2_j=\frac1{n}\sum_{i=1}^n {\tilde Z}_{ij}^2 \hat\pi
/(1-\hat\pi)^2.$ This modification introduces in the bounds an
additional term proportional to $\hat\pi - \pi_*$, which is of the
order $O((np)^{-1/2})$ in probability and hence is negligible as
compared to the error bound for the Compensated MU selector.

\begin{remark}
In this section, we have considered non-random $X$. Using the same
argument, it is easy to derive analogous expressions for
$\sigma_i(\e)$ and $b(\e)$ when $X$ is a random matrix with
independent sub-gaussian entries, and $\xi$, $\Xi$ are independent
from $X$.
\end{remark}

\section{Confidence intervals}\label{sec:confidence}

The bounds of Theorems \ref{t1} and \ref{t2} depend on the unknown
matrix $X$ via the sensitivities, and therefore cannot be used to
provide confidence intervals. In this section, we show how to
address the issue of confidence intervals by deriving other type of
bounds based on the empirical sensitivities. Note first that the
matrix $\widehat \Psi = \frac1{n}Z^TZ-\widehat D$ is a natural
estimator of the unknown Gram matrix $\Psi$. It is
$\sqrt{n}$-consistent in $\ell_\infty$-norm under the conditions of
the previous section. Therefore, it makes sense to define the
empirical counterparts of $\kappa_{q}(s)$ and $\kappa_{k}^*(s)$ by
the relations:
$$
\hat\kappa_{q}(s)\triangleq\min_{J: \ |J|\le s}
\left(\min_{\Delta\in C_J:\ |\Delta|_q=1} |\widehat \Psi\Delta
|_{\infty}\right),
$$
and
\begin{equation*}
\hat\kappa_{k}^*(s)\triangleq \min_{J: \ |J|\le s}
\left(\min_{\Delta\in C_J: \ \Delta_k= 1} |\widehat\Psi\Delta
|_{\infty}\right).
\end{equation*}
The values $\hat\kappa_{q}(s)$ and $\hat\kappa_{k}^*(s)$ that we
will call the {\em empirical sensitivities} can be efficiently
computed for small $s$ or, alternatively, one can compute
data-driven lower bounds on them for any $s$ using linear
programming, cf. Gautier and Tsybakov~(2011).

The following theorem establishes confidence intervals for
$s$-sparse vector $\theta^*$ based on the empirical sensitivities.

\begin{theorem}\label{t3}
Assume that model (\ref{0})--(\ref{00}) is valid with an $s$-sparse
vector of parameters $\t^*\in \Th$, where $\Th$ is a given subset of
$\mathbb{R}^p$. Then, with probability at least $1-6\varepsilon$,
 for any solution
$\hat\t$ of \eqref{mu_sel2} we have
\begin{eqnarray}\label{t3:1}
&&|\hat\t-\t^*|_q\le \frac{2(\mu(\varepsilon)|\hat\theta|_1
+\tau(\e))}{\hat\kappa_q(s)(1-\mu(\e)/\hat\kappa_1(s))_+}\,, \quad
\forall \ 1\le q\le \infty,
\\
&&|\hat\t_k-\t^*_k| \le \frac{2(\mu(\varepsilon)|\hat\theta|_1
+\tau(\e))}{\hat\kappa_k^*(s)(1-\mu(\e)/\hat\kappa_1(s))_+}\,,\quad
\forall \ 1\le k \le p, \label{t3:2}
\end{eqnarray}
where $x_+=\max(0,x)$, and we set $1/0\triangleq\infty$.
\end{theorem}

\begin{proof} Set $\Delta=\t^*-\hat\t$, and write for brevity
$S(\theta)=\frac1{n}Z^T(y-Z\theta)+ \hat D \theta.$ Using
Lemma~\ref{lem3} in Section~\ref{sec:preuve}, the fact that
$|\Delta_{J^c}|_1 \le |\Delta_{J}|_1$ where $J$ is the set of
non-zero components of $\t^*$ (cf. Lemma~1 in Rosenbaum and
Tsybakov~(2010)) and the definition of the empirical sensitivity
$\hat\kappa_1(s)$, we find
\begin{eqnarray*}
|\hat \Psi \Delta|_\infty &\le&
|S(\t^*)|_\infty+|S(\hat\t)|_\infty \\
&\le& \mu(\varepsilon)(|\theta^*|_1+|\hat\theta|_1) + 2\tau(\e)\\
&\le&2(\mu(\varepsilon)|\hat\theta|_1 +\tau(\e)) +
\mu(\varepsilon)|\Delta|_1
\\
&\le&2(\mu(\varepsilon)|\hat\theta|_1 +\tau(\e)) +
\frac{\mu(\varepsilon)}{\hat\kappa_1(s)}|\hat \Psi \Delta|_\infty
\end{eqnarray*}
This and the definition of $\hat\kappa_q(s)$ yield (\ref{t3:1}). The
proof of (\ref{t3:2}) is analogous, with $\hat\kappa_k^*(s)$ used
instead of $\hat\kappa_q(s)$.
\end{proof}

\begin{remark}
Note that the bounds (\ref{t3:1})--(\ref{t3:2}) remain valid for
$s'\geq s$. Therefore, if one gets an estimator $\hat{s}$ of $s$
such that $\hat{s}\geq s$ with high probability, it can be plugged
in into the bounds in order to get completely feasible confidence
intervals.
\end{remark}

\section{Simulations}\label{simu}
We consider here the model with missing data (\ref{ex1}).
Simulations in Rosenbaum and Tsybakov~(2010) indicate that in this
model the MU selector achieves better numerical performance than the
Lasso or the Dantzig selector. Here we compare the MU selector with
the Compensated MU selector.
%
%
We design the numerical experiment the following way.

\noindent $-$ We take a matrix $X$ of size $100 \times 500$ ($n=100,
p=500$) which is the normalized version (centered and then
normalized so that all the diagonal elements of the associated Gram
matrix $X^{T}X/n$ are equal to 1) of a $100
\times 500$ matrix with i.i.d. standard Gaussian entries.\\
\noindent $-$ For a given integer $s$, we randomly (uniformly)
choose $s$ non-zero elements in a vector $\theta^*$ of size $500$.
The associated coefficients $\theta_j^*$ are set to $0.5$, and all
other
coefficients are set to 0. We take $s=1,2,3,5,10$.\\
\noindent $-$ We set $y=X\theta^*+\xi$, where $\xi$ a vector with
i.i.d. zero mean and variance $\nu^2$ normal components, $\nu =0.05/1.96$. \\
\noindent $-$ We compute the values $Z_{ij}={\tilde
Z}_{ij}/(1-\pi_*)$
 with ${\tilde Z}_{ij}$ as in \eqref{ex1} \footnote{Remark that this experiment slightly differs from those in
Rosenbaum and Tsybakov~(2010) where the matrix taken in
\eqref{mu_sel} has entries ${\tilde Z}_{ij}$.}, and $\pi_j=
0.1\triangleq \pi_*$ for all $j$. (The value $\pi_*$ rather than its
empirical counterpart, which is very close to $\pi_*$, is used in
the
algorithm to simplify the computations).\\
\noindent $-$ We run a linear programming algorithm to compute the
solutions of \eqref{mu_sel} and \eqref{mu_sel2} where we optimize
over $\Theta =\mathbb{R}_{+}^{500}$. To simplify the comparison with
Rosenbaum and Tsybakov~(2010), we write $\mu$ in the form
$(1+\delta)\delta$ with $\delta=0, 0.01, 0.05, 0.075, 0.1$. In particular,
$\delta=0$ corresponds to the Dantzig selector based on the noisy
matrix~$Z$. In practice, one can use an empirical procedure of the
choice of $\delta$ described in Rosenbaum and Tsybakov~(2010). The
choice of $\tau$ is not crucial and influences only slightly the
output of the algorithm. The results presented below correspond to
$\tau$ chosen in the same way as in the numerical study in Rosenbaum
and Tsybakov~(2010).
\\
\noindent $-$ We compute the error measures
$$\text{Err}_1=|\hat{\theta}-\theta^*|_2^2\text{ and }\text{Err}_2=
|X(\hat{\theta}-\theta^*)|_2^2.$$ We also record the retrieved
sparsity pattern, which is defined as the set of the non-zero
coefficients of $\hat{\theta}$.\\
\noindent $-$ For each value of $s$ we run $100$ Monte Carlo
simulations.

Tables 1--5 present the empirical averages and standard deviations
(in brackets) of $\text{Err}_1$, $\text{Err}_2$, of the number of
non-zero coefficients in $\hat{\theta}$ ($Nb_1$) and of the number
of non-zero coefficients in $\hat{\theta}$ belonging to the true
sparsity pattern ($Nb_2$). We also present the total number of
simulations where the sparsity pattern is exactly retrieved (Exact).
The lines with ``$\delta=v$" for $v=0, 0.01, 0.05, 0.075, 0.1$ correspond
to the MU selector and those with ``$C-\delta=v$" to the Compensated
MU selector.
\\
\setcounter{figure}{0}
\renewcommand{\figurename}{Tab.}
\begin{center}
\begin{small}
\begin{tabular}{|c|c|c|c|c|c|}
\hline
&$\text{Err}_1$&$\text{Err}_2$&$\text{Nb}_1$&$\text{Nb}_2$&Exact\\\hline
$\delta=0$ &$\underset{(0.0114)}{0.0196}$&$\underset{(0.5865)}{1.334}$&$\underset{(10.91)}{70.13}$&$\underset{(0)}{1}$&$0$\\
C-$\delta=0$ &$\underset{(0.0145)}{0.0225}$&$\underset{(0.6993)}{1.495}$&$\underset{(8.343)}{80.09}$&$\underset{(0)}{1}$&$0$\\\hline
$\delta=0.01$ &$\underset{(0.0069)}{0.0131}$&$\underset{(0.3606)}{0.9318}$&$\underset{(9.507)}{45.45}$&$\underset{(0)}{1}$&$1$\\
C-$\delta=0.01$ &$\underset{(0.0062)}{0.0095}$&$\underset{(0.4625)}{0.8386}$&$\underset{(9.737)}{46.88}$&$\underset{(0)}{1}$&$0$\\\hline
$\delta=0.05$ &$\underset{(0.0038)}{0.0100}$&$\underset{(0.2121)}{0.8001}$&$\underset{(5.798)}{12.45}$&$\underset{(0)}{1}$&$3$\\
C-$\delta=0.05$&$\underset{(0.0027)}{0.0042}$&$\underset{(0.1844)}{0.3412}$&$\underset{(5.764)}{10.52}$&$\underset{(0)}{1}$&$6$\\\hline
$\delta=0.075$ &$\underset{(0.0030)}{0.0100}$&$\underset{(0.1869)}{0.8878}$&$\underset{(4.261)}{6.28}$&$\underset{(0)}{1}$&$14$\\
C-$\delta=0.075$&$\underset{(0.0020)}{0.0038}$&$\underset{(0.1348)}{0.3377}$&$\underset{(3.674)}{4.91}$&$\underset{(0)}{1}$&$21$\\\hline
$\delta=0.1$&$\underset{(0.0024)}{0.0110}$&$\underset{(0.1582)}{1.038}$&$\underset{(2.640)}{3.22}$&$\underset{(0)}{1}$&$36$\\
C-$\delta=0.1$&$\underset{(0.0015)}{0.0044}$&$\underset{(0.1040)}{0.4255}$&$\underset{(2.042)}{2.37}$&$\underset{(0)}{1}$&$54$\\\hline
\end{tabular}
\end{small}
\caption{Results for the model with missing data, $s=1$.}
\end{center}

\begin{center}
\begin{small}
\begin{tabular}{|c|c|c|c|c|c|}
\hline
&$\text{Err}_1$&$\text{Err}_2$&$\text{Nb}_1$&$\text{Nb}_2$&Exact\\\hline
$\delta=0$ &$\underset{(0.0170)}{0.0437}$&$\underset{( 1.060)}{2.756}$&$\underset{(5.149)}{80.04}$&$\underset{(0)}{2}$&$0$\\
C-$\delta=0$ &$\underset{(0.0275)}{0.0685}$&$\underset{(1.129)}{2.951}$&$\underset{(3.911)}{92.67}$&$\underset{(0)}{2}$&$0$\\\hline
$\delta=0.01$ &$\underset{(0.0107)}{0.0287}$&$\underset{(0.5423)}{1.838}$&$\underset{(6.717)}{49.29}$&$\underset{(0)}{2}$&$0$\\
C-$\delta=0.01$ &$\underset{(0.0098)}{0.0201}$&$\underset{(0.6827)}{1.561}$&$\underset{(6.775)}{48.18}$&$\underset{(0)}{2}$&$0$\\\hline
$\delta=0.05$ &$\underset{(0.0093)}{0.0264}$&$\underset{(0.4960)}{2.105}$&$\underset{(4.631)}{10.35}$&$\underset{(0)}{2}$&$1$\\
C-$\delta=0.05$&$\underset{(0.0066)}{0.0125}$&$\underset{(0.3849)}{0.9796}$&$\underset{(4.092)}{7.70}$&$\underset{(0)}{2}$&$8$\\\hline
$\delta=0.075$ &$\underset{(0.0090)}{0.0301}$&$\underset{(0.5022)}{2.694}$&$\underset{(2.587)}{4.77}$&$\underset{(0)}{2}$&$24$\\
C-$\delta=0.075$&$\underset{(0.0052)}{0.0148}$&$\underset{(0.3573)}{1.359}$&$\underset{(1.924)}{3.41}$&$\underset{(0)}{2}$&$47$\\\hline
$\delta=0.1$&$\underset{(0.0086)}{0.0371}$&$\underset{(0.4730)}{3.521}$&$\underset{(1.046)}{2.62}$&$\underset{(0)}{2}$&$65$\\
C-$\delta=0.1$&$\underset{(0.0059)}{0.0218}$&$\underset{(0.3853)}{2.088}$&$\underset{(0.617)}{2.28}$&$\underset{(0)}{2}$&$77$\\\hline
\end{tabular}
\end{small}
\caption{Results for the model with missing data, $s=2$.}
\end{center}

\begin{center}
\begin{small}
\begin{tabular}{|c|c|c|c|c|c|}
\hline
&$\text{Err}_1$&$\text{Err}_2$&$\text{Nb}_1$&$\text{Nb}_2$&Exact\\\hline
$\delta=0$ &$\underset{(0.0296)}{0.0772}$&$\underset{(1.268)}{4.361}$&$\underset{(4.177)}{83.95}$&$\underset{(0)}{3}$&$0$\\
C-$\delta=0$ &$\underset{(0.0436)}{0.1480}$&$\underset{(1.253)}{4.258}$&$\underset{(3.262)}{97.76}$&$\underset{(0)}{3}$&$0$\\\hline
$\delta=0.01$ &$\underset{(0.0176)}{0.0493}$&$\underset{(0.7907)}{2.929}$&$\underset{(6.515)}{49.78}$&$\underset{(0)}{3}$&$0$\\
C-$\delta=0.01$ &$\underset{(0.0153)}{0.0351}$&$\underset{(0.8442)}{2.328}$&$\underset{(6.302)}{48.23}$&$\underset{(0)}{3}$&$0$\\\hline
$\delta=0.05$ &$\underset{(0.0166)}{0.0528}$&$\underset{(0.7696)}{4.295}$&$\underset{(3.907)}{9.82}$&$\underset{(0)}{3}$&$1$\\
C-$\delta=0.05$&$\underset{(0.0109)}{0.0281}$&$\underset{(0.6360)}{2.343}$&$\underset{(3.608)}{7.02}$&$\underset{(0)}{3}$&$18$\\\hline
$\delta=0.075$ &$\underset{(0.0161)}{0.0643}$&$\underset{(0.7865)}{5.842}$&$\underset{(2.086)}{5.16}$&$\underset{(0)}{3}$&$29$\\
C-$\delta=0.075$&$\underset{(0.0106)}{0.0384}$&$\underset{(0.6556)}{3.606}$&$\underset{(1.177)}{3.82}$&$\underset{(0)}{3}$&$57$\\\hline
$\delta=0.1$&$\underset{(0.0164)}{0.0814}$&$\underset{(0.7434)}{7.792}$&$\underset{(0.9618)}{3.57}$&$\underset{(0)}{3}$&$64$\\
C-$\delta=0.1$&$\underset{(0.0121)}{0.0575}$&$\underset{(0.6554)}{5.538}$&$\underset{(0.3912)}{3.13}$&$\underset{(0)}{3}$&$89$\\\hline
\end{tabular}
\end{small}
\caption{Results for the model with missing data, $s=3$.}
\end{center}

\begin{center}
\begin{small}
\begin{tabular}{|c|c|c|c|c|c|}
\hline
&$\text{Err}_1$&$\text{Err}_2$&$\text{Nb}_1$&$\text{Nb}_2$&Exact\\\hline
$\delta=0$ &$\underset{(0.0536)}{0.1470}$&$\underset{(1.686)}{6.801}$&$\underset{(3.683)}{87.35}$&$\underset{(0)}{5}$&$0$\\
C-$\delta=0$ &$\underset{(0.0802)}{0.3631}$&$\underset{(1.490)}{6.114}$&$\underset{(4.039)}{104.23}$&$\underset{(0)}{5}$&$0$\\\hline
$\delta=0.01$ &$\underset{(0.0340)}{0.0961}$&$\underset{(1.180)}{4.928}$&$\underset{(5.527)}{49.64}$&$\underset{(0)}{5}$&$0$\\
C-$\delta=0.01$ &$\underset{(0.0281)}{0.0670}$&$\underset{(1.206)}{3.627}$&$\underset{(6.298)}{46.69}$&$\underset{(0)}{5}$&$0$\\\hline
$\delta=0.05$ &$\underset{(0.0391)}{0.1375}$&$\underset{(1.557)}{11.100}$&$\underset{(3.347)}{10.34}$&$\underset{(0)}{5}$&$6$\\
C-$\delta=0.05$&$\underset{(0.0307)}{0.0864}$&$\underset{(1.475)}{7.302}$&$\underset{(2.404)}{7.42}$&$\underset{(0)}{5}$&$27$\\\hline
$\delta=0.075$ &$\underset{(0.0427)}{0.1769}$&$\underset{(1.548)}{15.68}$&$\underset{(1.867)}{6.85}$&$\underset{(0)}{5}$&$31$\\
C-$\delta=0.075$&$\underset{(0.0427)}{0.1311}$&$\underset{(1.737)}{11.86}$&$\underset{(1.013)}{5.55}$&$\underset{(0)}{5}$&$68$\\\hline
$\delta=0.1$&$\underset{(0.0455)}{0.2286}$&$\underset{(1.385)}{21.19}$&$\underset{(1.049)}{5.67}$&$\underset{(0)}{5}$&$58$\\
C-$\delta=0.1$&$\underset{(0.0595)}{0.1933}$&$\underset{(2.056)}{17.71}$&$\underset{(0.6114)}{5.19}$&$\underset{(0)}{5}$&$88$\\\hline
\end{tabular}
\end{small}
\caption{Results for the model with missing data, $s=5$.}
\end{center}

\begin{center}
\begin{small}
\begin{tabular}{|c|c|c|c|c|c|}
\hline
&$\text{Err}_1$&$\text{Err}_2$&$\text{Nb}_1$&$\text{Nb}_2$&Exact\\\hline
$\delta=0$ &$\underset{(0.1407)}{0.4479}$&$\underset{(3.060)}{14.56}$&$\underset{(2.881)}{92.21}$&$\underset{(0)}{10}$&$0$\\
C-$\delta=0$ &$\underset{(0.1705)}{1.208}$&$\underset{(2.197)}{11.90}$&$\underset{(6.532)}{117.23}$&$\underset{(0)}{10}$&$0$\\\hline
$\delta=0.01$ &$\underset{(0.1263)}{0.3512}$&$\underset{(1.997)}{13.59}$&$\underset{(5.340)}{52.76}$&$\underset{(0)}{10}$&$0$\\
C-$\delta=0.01$ &$\underset{(0.1317)}{0.2921}$&$\underset{(2.049)}{10.70}$&$\underset{(6.067)}{48.74}$&$\underset{(0)}{10}$&$0$\\\hline
$\delta=0.05$ &$\underset{(0.2395)}{0.7660}$&$\underset{(4.389)}{47.13}$&$\underset{(4.152)}{20.29}$&$\underset{(0.1959)}{9.96}$&$0$\\
C-$\delta=0.05$&$\underset{(0.2696)}{0.6919}$&$\underset{(5.709)}{41.55}$&$\underset{(4.241)}{16.99}$&$\underset{(0.2374)}{9.94}$&$1$\\\hline
$\delta=0.075$ &$\underset{(0.2721)}{0.9683}$&$\underset{(5.496)}{65.24}$&$\underset{(3.545)}{16.78}$&$\underset{(0.4092)}{9.85}$&$0$\\
C-$\delta=0.075$&$\underset{(0.3067)}{0.9443}$&$\underset{(7.066)}{61.23}$&$\underset{(3.452)}{15.00}$&$\underset{(0.5499)}{9.76}$&$5$\\\hline
$\delta=0.1$&$\underset{(0.2807)}{1.150}$&$\underset{(6.745)}{82.86}$&$\underset{(2.948)}{14.84}$&$\underset{(0.6508)}{9.58}$&$1$\\
C-$\delta=0.1$&$\underset{(0.3049)}{1.157}$&$\underset{(8.359)}{80.43}$&$\underset{(2.804)}{13.57}$&$\underset{(0.7601)}{9.39}$&$11$\\\hline
\end{tabular}
\end{small}
\caption{Results for the model with missing data, $s=10$.}
\end{center}

The results of the simulations are quite convincing. Indeed, the
Compensated MU selector improves upon the MU selector with respect
to all the considered criteria, in particular when $\theta^*$ is
very sparse ($s=1,2,3$). The order of magnitude of the improvement
is such that, for the best $\delta$, the errors $\text{Err}_1$ and
$\text{Err}_2$ are divided by $2$. The improvement is not so
significant for larger $s$, especially for $s=10$ when the model
starts to be not very sparse.  For all the values of $s$, the
non-zero coefficients of $\theta^*$ are systematically in the
sparsity pattern both of the MU selector and of the Compensated MU
selector. The total number of non-zero coefficients is always
smaller (i.e., closer to the correct one) for the Compensated MU
selector. Finally, note that the best results for the error measures
$\text{Err}_1$ and $\text{Err}_2$ are obtained with $\delta\le
0.075$, while the sparsity pattern is better retrieved for
$\delta=0.1$. This reflects a trade-off between estimation and
selection.

\section{Proofs}\label{sec:preuve}

{\bf Proof of Remark \ref{rem:1}.} It is enough to show that
$\mathcal{A}(\mu,\tau)=\mathcal{B}(\mu,\tau)$ where
$$
\mathcal{B}(\mu,\tau)=\{\theta\in\Theta: \ \exists \
u\in\mathbb{R}^p \ \text{such that} \ (\theta,u)\in W(\mu,\tau)\}.
$$
Let first $(\t,u)\in W(\mu,\tau)$. Using the triangle inequality, we
easily get that $\t\in\mathcal{A}(\mu,\tau)$. Now take
$\t\in\mathcal{A}(\mu,\tau)$. We set
$$N=\frac{1}{n}Z^T(y-Z\theta)+\widehat D\theta$$
 and consider $u\in\mathbb{R}^p$ defined by
$$u_i=-N_i\mathrm{1}_{\{|N_i|\leq \mu|\theta|_1\}}
-\text{sign}(N_i)\mu|\theta|_1\mathrm{1}_{\{|N_i|>\mu|\theta|_1\}},$$
for $i=1,\ldots,p$, where $u_i$ and $N_i$ are the $i$th components
of $u$ and $N$ respectively. It is easy to check that $(\theta,u)\in
W\big(\mu,\tau\big)$, which concludes the proof.

\smallskip

\noindent {\bf Proof of Theorem \ref{t1}.} The proof is based on two
lemmas. For brevity, we will skip the dependence of $b(\e),
\delta_i(\e)$ and $\nu(\e)$ on $\varepsilon$.
\begin{lemma}\label{lem3}
With probability at least $1-6\varepsilon$, we have
$\theta^*\in\mathcal{A}(\varepsilon)$.
\end{lemma}
\begin{proof}
We first write that
$Z^T(y-Z\theta^*)+n{\widehat D}\theta^*$ is equal to
\begin{multline*}
-X^T\Xi\theta^*+X^T\xi+\Xi^T\xi-(\Xi^T\Xi-\text{Diag}\{\Xi^T\Xi\})\theta^*\\
-(\text{Diag}\{\Xi^T\Xi\}-nD)\theta^*+n(\widehat D-D)\theta^*.
\end{multline*}
By definition of the $\delta_i(\varepsilon)$ and $b(\varepsilon)$, with probability at least
$1-6\varepsilon$ we have
\begin{align}
&|\frac{1}{n}X^T\Xi\theta^*|_{\infty}\leq|\frac{1}{n}X^T\Xi|_{\infty}|\theta^*|_1\leq\delta_1|\theta^*|_1\label{ineg}\\
&|\frac{1}{n}X^T\xi|_{\infty}+|\frac{1}{n}\Xi^T\xi|_{\infty}\leq \delta_2+\delta_3\\
&|\frac{1}{n}(\Xi^T\Xi-\text{Diag}\{\Xi^T\Xi\})\theta^*|_{\infty}\leq|\frac{1}{n}(\Xi^T\Xi-\text{Diag}\{\Xi^T\Xi\})|_{\infty}|\theta^*|_1\leq\delta_4|\theta^*|_1\\
&|(\frac{1}{n}\text{Diag}\{\Xi^T\Xi\}-D)\theta^*|_{\infty}\leq|\frac{1}{n}\text{Diag}\{\Xi^T\Xi\}-D|_{\infty}|\theta^*|_1\leq \delta_5|\theta^*|_1\\
&|(\widehat D-D)\theta^*|_{\infty}\leq b|\theta^*|_1\label{ineg2}.
\end{align}
Therefore  $\theta^*\in\mathcal{A}(\varepsilon)$ with probability at
least $1-6\varepsilon$.
\end{proof}

\begin{lemma}\label{lem4}
With probability at least $1-6\varepsilon$, for
$\Delta=\hat{\theta}-\theta^*$ we have
$$|\frac{1}{n}X^TX\Delta|_{\infty}\leq \nu.$$
\end{lemma}
\begin{proof}
Throughout the proof, we assume that we are on event of probability
at least $1-6\varepsilon$ where inequalities \eqref{ineg} --
\eqref{ineg2} hold and $\theta^*\in\mathcal{A}(\varepsilon)$. We
have
$$|\frac{1}{n}X^TX\Delta|_{\infty}\leq |\frac{1}{n}Z^T(Z\hat{\theta}-\Xi\hat{\theta}-y+\xi)|_{\infty}+|\frac{1}{n}\Xi^TX\Delta|_{\infty}.$$
Consequently,
\begin{multline*}
|\frac{1}{n}X^TX\Delta|_{\infty}\leq |\frac{1}{n}Z^T(Z\hat{\theta}-y)-{\widehat D}\hat{\theta}|_{\infty}\\
+|(\frac{1}{n}Z^T\Xi-D)\hat{\theta}|_{\infty}+|(\widehat
D-D)\hat{\theta}|_{\infty}+|\frac{1}{n}Z^T\xi|_{\infty}+|\frac{1}{n}\Xi^TX\Delta|_{\infty}.
\end{multline*}
Using that $\hat{\theta}\in\mathcal{A}(\varepsilon)$, we easily get
that $|\frac{1}{n}X^TX\Delta|_{\infty}$ is not greater than
$$\mu|\hat\theta|_1+2\delta_2+2\delta_3+b|\hat\theta|_1+|(\frac{1}{n}Z^T\Xi-D)\hat{\theta}|_{\infty}+|\frac{1}{n}\Xi^TX\Delta|_{\infty}.
$$
Now remark that
\begin{eqnarray*}
&&|(\frac{1}{n}Z^T\Xi-D)\hat{\theta}|_{\infty}\leq
|\frac{1}{n}Z^T\Xi-D|_{\infty}|\hat{\theta}|_1 \\&& \le
\big(|\frac{1}{n}(\Xi^T\Xi-\text{Diag}\{\Xi^T\Xi\})|_{\infty}
+|\frac{1}{n}\text{Diag}\{\Xi^T\Xi\}-D|_{\infty}
+|\frac{1}{n}X^T\Xi|_{\infty}\big)|\hat{\theta}|_1\\
&& \le (\delta_1+\delta_4+\delta_5)|\hat{\theta}|_1.
\end{eqnarray*}
Finally, using that
$$
 |\frac{1}{n}\Xi^TX\Delta|_{\infty}\leq |\hat{\theta}-\theta^*|_1|\frac{1}{n}X^T\Xi|_{\infty}\leq \delta_1(|\hat{\theta}|_1+|\theta^*|_1)\\
$$
together with the fact that $|\hat{\theta}|_1\leq |\theta^*|_1$, we
obtain the result.
\end{proof}

\smallskip

We now proceed to the proof of Theorem \ref{t1}. The bounds
(\ref{t1:1}) and (\ref{t1:2}) follow from Lemma \ref{lem4}, the fact
that $|\Delta_{J^c}|_1 \le |\Delta_{J}|_1$ where $J$ is the set of
non-zero components of $\t^*$ (cf. Lemma~1 in Rosenbaum and
Tsybakov~(2010)) and the definition of the sensitivities
$\kappa_q(s)$, $\kappa_k^*(s)$. To prove (\ref{t1:3}), first note
that
\begin{equation}\label{p1}
\frac{1}{n}|X\Delta|_2^2\le
\frac{1}{n}|X^TX\Delta|_{\infty}|\Delta|_1,
\end{equation}
and use (\ref{t1:1}) with $q=1$ and Lemma \ref{lem4}. This yields
the first term under the minimum on the right hand side of
(\ref{t1:3}). The second term is obtained again from (\ref{p1}),
Lemma~\ref{lem4} and the inequality $|\Delta|_1\le |\hat\t|_1 +
|\t^*|_1\le 2|\t^*|_1$.

\smallskip

\noindent {\bf Proof of Theorem \ref{t2}.} The bounds (\ref{t2:1})
and (\ref{t2:4}) follow by combining (\ref{t1:1}) with (\ref{k3})
and with (\ref{k1}) -- (\ref{k2}) respectively. Next, (\ref{t2:2})
follows from (\ref{t1:3}) and (\ref{k3}). Also, as an easy
consequence of (\ref{t1:1}) and (\ref{k4}) with $q=2$ we get
$$
|\hat\t-\t^*|_2 \le \frac{4\nu(\varepsilon)s^{1/2}}{\kappa_{\rm
RE}(2s)}\,.
$$
Finally, (\ref{t2:3}) follows from this inequality and (\ref{t2:1})
using the interpolation formula $|\Delta|_q^q\le
|\Delta|_1^{2-q}|\Delta|_2^{2(q-1)}$ for $\Delta = \hat\t-\t^*$, and
the fact that $\kappa_{\rm RE}(s)\ge \kappa_{\rm RE}(2s)$.


\begin{thebibliography}{99}


\bibitem{BC2} Belloni, A., and Chernozhukov, V. (2011).
High dimensional sparse econometric models: an introduction. In:
{\em Inverse Problems and High Dimensional Estimation, Stats in the
Ch\^ateau 2009}, Alquier, P., E. Gautier, and G. Stoltz, Eds., {\em
Lecture Notes in Statistics}, {\bf 203} 127--162, Springer, Berlin.

\bibitem{brt} Bickel, P.J., Ritov, Y. and Tsybakov, A.B. (2009).
Simultaneous analysis of Lasso and Dantzig selector. The Annals of
Statistics {\bf 37} 1705-1732.

\bibitem{btw07a}
Bunea, F., Tsybakov, A.B. and Wegkamp, M.H. (2007a). Aggregation for
Gaussian regression. The Annals of Statistics {\bf 35} 1674-1697.


\bibitem{btw07b}
Bunea, F., Tsybakov, A.B. and Wegkamp, M.H. (2007b). Sparsity oracle
inequalities for the Lasso, {Electronic Journal of Statistics} {\bf
1}  169-194.

\bibitem{BvdG}
{B\"uhlmann, P., and S. A. van de Geer} (2011). {\em Statistics for
High-Dimensional Data}. Springer, New-York.

%
%
%

\bibitem {ct} Cand\`es, E.J. and Tao, T. (2007). The Dantzig selector: statistical
estimation when $p$ is much larger than $n$ (with discussion). The
Annals of Statistics {\bf 35}  2313-2404.


%
%
%


\bibitem{gt} Gautier, E. and Tsybakov, A.B~(2011) High-dimensional instrumental variables regression and
confidence sets. \texttt{arXiv:1105.2454}

\bibitem{l} Lounici, K. (2008). Sup-norm
convergence rate and sign concentration property of Lasso and
Dantzig estimators. Electronic Journal of Statistics {\bf 2}
90-102.


\bibitem{kol} Koltchinskii, V. (2009). Dantzig selector and sparsity oracle
inequalities. Bernoulli {\bf 15} 799-828.

\bibitem{kolsf} Koltchinskii, V. (2011). Oracle inequalities in empirical risk minimization
and sparse recovery problems. \'Ecole d'\'Et\'e de Probabilit\'es de
Saint-Flour 2008. Lecture Notes in Mathematics, Vol. 2033.



%


\bibitem{p} Petrov,V.V. (1995). {\em Sums of Independent Random
Variables.} Oxford University Press.

\bibitem{rt} Rosenbaum, M. and Tsybakov A.B. (2010). Sparse
recovery under matrix uncertainty. The Annals of Statistics {\bf 38}
2620--2651.



%
%
%
%

\end{thebibliography}
\end{document}